\documentclass[11pt]{amsart}

\usepackage{amssymb,amsmath,amsthm,enumerate}
\numberwithin{equation}{section}
\newtheorem{thm}{Theorem}[section]
\newtheorem{lem}[thm]{Lemma}
\newtheorem{cor}[thm]{Corollary}
\newtheorem{prop}[thm]{Proposition}
\theoremstyle{remark}
\newtheorem*{definition}{Definition}

\newtheorem{rem}[thm]{Remark}

\newcommand{\hp}{\mathbb{H}}
\newcommand{\Z}{\mathbb{Z}}

\newcommand{\Q}{\mathbb{Q}}
\newcommand{\F}{\mathbb{F}}

\newcommand{\N}{\mathbb{N}}

\renewcommand{\mod}[1]{\,(\mathrm{mod}\,{#1})}

\newcommand{\leg}[2]{\left(\frac{#1}{#2}\right)}

\newcommand{\ord}{\mathrm{ord}}

\title{Non-existence of Ramanujan congruences in modular forms of level four}
\author{Michael Dewar}

\begin{document}

\begin{abstract}
Ramanujan famously found congruences for the partition function like $p(5n+4)\equiv 0 \mod 5$.  We provide a method to find all simple congruences of this type in the coefficients of the inverse of a modular form on $\Gamma_{1}(4)$ which is non-vanishing on the upper half plane.  This is applied to answer open questions about the (non)-existence of congruences in the generating functions for overpartitions, crank differences, and 2-colored $F$-partitions.
\end{abstract}

\maketitle

\section{Introduction}
Define the partition function $p(n)$ to be the number of ways of writing $n$ as a sum of non-increasing positive integers.  Ramanujan famously found congruences for the partition function
\begin{align}
\begin{split}\label{ramcong}
p(5n+4) &\equiv 0 \mod 5\\
p(7n+5) &\equiv 0 \mod 7\\
p(11n+6) &\equiv 0 \mod {11}
\end{split}
\end{align}
and raised the question of whether there are other primes $\ell$ for which
\begin{equation*}
p(\ell n + b) \equiv 0 \mod \ell
\end{equation*}
for some $b\in\Z$.  We refer to congruences of this form as {\it Ramanujan congruences}.  Kiming and Olsson \cite{kimols1992a} use the Tate cycles of the $\Theta$ operator to show that the parameters for any such congruence for $p(n)$ must satisfy $24b\equiv 1 \mod \ell$.  Ahlgren and Boylan \cite{ahlboy2003a} build on this result to prove that (\ref{ramcong}) are the only Ramanujan congruences of the partition function.  The existence of non-Ramanujan congruences of the partition function is shown by Ono~\cite{ono2000a} and the existence of Ramanujan congruences in powers of the partition generating function is studied by Boylan~\cite{boy2004a}.  In this paper we provide a general method for investigating sequences related to modular forms and prove the non-existence of Ramanujan congruences (for large primes $\ell$) in three well-known combinatorial objects.

Andrews~\cite{and1984a} introduces {\it generalized Frobenius partitions}, also called {\it $F$-partitions}, in which a number $n$ is represented as
\begin{align*}
n = r + \sum_{i=1}^{r} a_{i} +  \sum_{i=1}^{r} b_{i}
\end{align*}
where $\{a_{i}\}$ and $\{b_{i}\}$ are both strictly decreasing sequences of non-negative integers.  An $F$-partition is often represented as
\begin{align*}
\begin{pmatrix}
a_{1} & a_{2} & \cdots & a_{r}\\
b_{1} & b_{2} & \cdots & b_{r}
\end{pmatrix}.
\end{align*}
An $F$-partition is {\it $2$-colored} if it is constructed from $2$ copies of the non-negative integers, written $j_{i}$ with $j\geq 0$ and $i=1,2$. Say $j_{i} < r_{s}$ if $j<r$ or both $j=r$ and $i<s$.  Let $c\phi_{2}(n)$ denote the number of 2-colored $F$-partitions of $n$.  Andrews (\cite{and1984a} Corollary~10.1 and Theorem~10.2) shows
\begin{align}
c\phi_{2}(2n+1) &\equiv 0\mod 2\label{andrews1}\\
c\phi_{2} (5n+3) &\equiv 0 \mod 5.\label{andrews2}
\end{align}
Eichhorn and Sellers \cite{eicsel2002a} prove $c\phi_{2}(5^{\alpha}n+\lambda_{\alpha})\equiv 0 \mod{5^{\alpha}}$ where $\lambda_{\alpha}$ is the least positive reciprocal of 12 modulo $5^{\alpha}$ and $\alpha=1,2,3,$ or $4$.  Recent work of Paule and Radu \cite{paurad2009a} settles the situation for all $\alpha\geq 5$.  Ono~\cite{ono1996a} and Lovejoy~\cite{lov2000a} use the theory of modular forms to prove the existence of certain congruences in $c\phi_{3}(n)$.  We prove
\begin{thm}\label{cphi2thm}
The only Ramanujan congruences $c\phi_{2} (\ell n+a)\equiv 0\mod\ell$ are (\ref{andrews1}) and (\ref{andrews2}).
\end{thm}

An {\it overpartition} of $n$ is a sum of non-increasing positive integers in which the first occurrence of an integer may be overlined.  Let $\overline{p}(n)$ count the number of such overpartitions and set $\overline{P}(z) = \sum \overline{p}(n)q^{n}$.  Background for overpartitions can be found in Corteel and Lovejoy\cite{corlov2004a}.  Recently, Mahlburg~\cite{mah2004a} has shown that the set of integers $n$ with $\overline{p}(n)\equiv 0 \mod {64}$ has arithmetic density 1, and Kim~\cite{kim2008b} has extended this result to modulus $128$.  For larger primes we have a very different situation.

\begin{thm}\label{overpartthm}
There are no Ramanujan congruences $\overline{p}(\ell n+ a)\equiv 0\mod \ell$ when $\ell\geq 3$.
\end{thm}

If $\pi$ is a (non-overlined) partition, define the {\it crank} by
\begin{align*}
\text{crank}(\pi) :=
\begin{cases}
\pi_{1} & \text{if } \mu(\pi) =0,\\
\nu(\pi) - \mu(\pi) & \text{if }\mu(\pi) > 0,
\end{cases}
\end{align*}
where $\pi_{1}$ denotes the largest part of $\pi$, $\mu(\pi)$ denotes the number of ones in $\pi$ and $\nu(\pi)$ denotes the number of parts of $\pi$ that are strictly larger than $\mu(\pi)$.  The existence of non-Ramanujan congruences for the crank counting function is proven by Mahlburg~\cite{mah2005a}.  Let $M_{e}(n)$ and $M_{o}(n)$ denote the number of partitions of $n$ with even and odd crank, respectively.  Choi, Kang, and Lovejoy~\cite{chokanlov2009a} study the crank difference function $(M_{e}-M_{o})(n)$ and find a Ramanujan congruence at $(M_{e}-M_{o})(5n +4)\equiv 0 \mod 5$.  They ask if the methods of \cite{kimols1992a} and \cite{ahlboy2003a} may be adapted to prove there are no other Ramanujan congruences.  We give a partial answer to their question.

\begin{thm}\label{crankthm}
Let $\ell \geq 5$ be prime, $\delta:=\frac{\ell^{2}-1}{24}$ and $a\not\equiv -\delta\mod\ell$.  The crank difference function has the Ramanujan congruence $(M_{e}-M_{o})(\ell n - \delta)\equiv 0\mod \ell$ if and only if $\ell=5$.  If $(M_{e}-M_{o})(\ell n + a)\equiv 0 \mod\ell$, then for all $b$ satisfying $\leg{a+\delta}{\ell}=\leg{b+\delta}{\ell}$, $(M_{e}-M_{o})(\ell n + b)\equiv 0 \mod\ell$.
\end{thm}

Theorems \ref{cphi2thm}, \ref{overpartthm} and \ref{crankthm} follow from a more general method of proving the non-existence of congruences for inverses of level four modular forms and their inverses.  Whereas Sinick~\cite{sin2009b} generalizes~\cite{kimols1992a} to prove that sequences of the form
\begin{align*}
\prod_{n=1}^{\infty} \prod_{i=1}^{2j} \frac{1}{1-q^{a_{i}n}} = \sum_{n=0}^{\infty} c(n) q^{n}
\end{align*}
admit only finitely many Ramanujan congruences, the extension to reciprocals of half-integral weight forms requires a generalization of~\cite{ahlboy2003a}.

\begin{thm}\label{algorithm_thm}
Let $\lambda \in \frac{1}{2}\Z$ be positive.  If $f\in M_\lambda (\Gamma_1(4)) \cap \Z [\![ q]\!]$ has no zeros in the upper half plane, then there are only finitely many primes $\ell$ for which the series $f^{-1} = \sum a(n)q^n \in \Z [\![ q]\!]$ has a Ramanujan congruence $a(\ell n + b) \equiv 0\mod\ell$.
\end{thm}
Moreover, we provide a method to find all of the Ramanujan congruences.  We provide two examples of Theorem~\ref{algorithm_thm}.  Let $\eta(z)=q^{1/24}\prod_{n=1}^{\infty}(1-q^{n})$ where $q=e^{2\pi i z}$.
\begin{thm}\label{mainthm}
Define $f:=\frac{\eta^6(z)\eta^6(4z)}{\eta^3(2z)}\in S_{9/2}(\Gamma_1(4))$ and let $f^{-1}=\sum a(n)q^n$.  The Ramanujan congruences of $f^{-1}$ are exactly
\begin{align*}
a(2n+0) &\equiv 0 \mod 2\\
a(3n+0) &\equiv 0 \mod 3\\
a(3n+1) &\equiv 0 \mod 3\\
a(5n+2) &\equiv 0 \mod 5\\
a(5n+3) &\equiv 0 \mod 5.
\end{align*}
\end{thm}

\begin{thm}\label{secondthm}
Define $f:=\frac{\eta^{14}(z)\eta^6(4z)}{\eta^{7}(2z)} \in S_{13/2}(\Gamma_1(4))$ and let $f^{-1}=\sum b(n)q^n$.  The Ramanujan congruences of $f^{-1}$ are exactly 
\begin{align*}
b(2n+0) &\equiv 0 \mod 2\\
b(7n+1) &\equiv 0 \mod 7\\
b(7n+2) &\equiv 0 \mod 7\\
b(7n+4) &\equiv 0 \mod 7.
\end{align*}
\end{thm}

Given $f\in M_\lambda (\Gamma_1(4)) \cap \Z [\![ q]\!]$ with $0<\lambda\in \frac{1}{2}\Z$, one may seek Ramanujan congruences for either $f$ or $f^{-1}$.  Each of these questions breaks into two cases depending on whether or not $\lambda$ is an integer.  We answer three of the four cases below:
\begin{table}[h]\centering
\begin{tabular}{|l | c | c |}\hline
  & $\lambda \in \Z$ & $\lambda \in \frac{1}{2}\Z - \Z$ \\\hline
Find congruences for $f$ & Open & Corollary~\ref{upper_half_of_range_half_wt_forms_cor} \\\hline
Find congruences for $f^{-1}$ & Corollary~\ref{inverses_of_integral_wt_forms_cor}&Section~\ref{alg_section} \\\hline
\end{tabular}
\end{table}

Corollaries~\ref{upper_half_of_range_half_wt_forms_cor} and \ref{inverses_of_integral_wt_forms_cor} provide explicit bounds on the possible primes $\ell$ for which there could be Ramanujan congruences.  Section~\ref{alg_section} provides a method to find all of the possible primes $\ell$ for which there could be such congruences.  One may then simply check the finitely many possibilities to generate a list of all Ramanujan congruences for the power series in question.  Seeking Ramanujan congruences in integral weight modular forms includes hard problems such as determining when Ramanujan's $\tau(n)$ function satisfies $\tau(\ell)\equiv 0\mod\ell$.  We leave such problems open.

We use the theory of Tate cycles for the reduction of modular forms $\mod\ell$.  In Section \ref{defnsection} we recall the basic machinery for reduced modular forms $\mod \ell$ and the $\Theta$ operator.  Since a modular form on $\Gamma_1(4)$ is completely determined by its zeros on $X_1(4)$, our guiding principle is to keep track of how the $\Theta$ operator changes (that is, increases) the orders at the cusps.  Section \ref{tcbasics} presents Jochnowitz's \cite{joc1982a} framework for analyzing Tate cycles.  Section \ref{KOsection} extends the work of Kiming and Olsson \cite{kimols1992a} to the level $4$ case.  In Section \ref{handybasissection} we lift $\mod\ell$ information to characteristic 0.  Section \ref{voacsection} works with modular forms which vanish only at the cusps.  By limiting our focus in this way, we ensure either $f$ or $\Theta f$ is a low point of the Tate cycle.  Thus, the filtrations of forms which are non-vanishing on the upper half plane will always provide a lower bound for the filtrations of the corresponding Tate cycle.  This lower bound is necessary in Section~\ref{alg_section} when we generalize the methods of Ahlgren and Boylan~\cite{ahlboy2003a} to prove Theorem~\ref{algorithm_thm}.  Finally, Section \ref{cooleg} proves the rest of the theorems.

For the sake of concreteness, we have chosen to work on $\Gamma_{1}(4)$, but the level is not an essential barrier.  In place of our Section~\ref{handybasissection}, one could instead appeal to the $q$-expansion principle and deduce a Sturm-style result by averaging over coset representatives for $\operatorname{SL}_{2}(\Z)/\Gamma_{1}(N)$ as in, for example, \cite{sin2009b}.

{\it Acknowledgements:} The author would like to thank Scott Ahlgren for all of his advice, support, and many thorough readings of this article.  The author is greatly indebted to Byungchan Kim for suggesting the examples in Theorems~\ref{cphi2thm}, \ref{overpartthm} and \ref{crankthm}.  Frank Garvan and Howard Skogman gave helpful suggestions.

\section{Reductions of modular forms and the $\Theta$ operator}\label{defnsection}
A more complete introduction to reductions of modular forms on $\operatorname{SL}_2(\Z)$ is found in \cite{swi1973a}.  The $\Gamma_1(4)=\left\{ \left[\begin{array}{cc}a & b \\c & d\end{array}\right]\in\operatorname{SL}_2(\Z) : a\equiv d \equiv 1\mod 4, c\equiv 0 \mod 4\right\}$ case is analogous.  Throughout this paper, $\ell \geq 5$ is prime, $q=e^{2\pi iz}$ and $M_k := M_k(\Gamma_1(4))\cap \Z_{(\ell)} [\![q]\!]$, where $\Z_{(\ell)}$ is the local ring $\{\frac{a}{b} \in \Q: \ell \nmid b\}$ and $0\leq k\in \Z$.  The space $M_{k}$ is the set of isobaric polynomials in
\begin{align}
F(z) &:= \sum_{n\geq 0} \sigma_1(2n+1)q^{2n+1} \in M_{2}\label{Fdefn}\\
\theta_0^{2}(z) &:= \left(\sum_{n\in\Z} q^{n^2}\right)^{2}\in M_{1}.\label{theta defn}
\end{align}
The expansions at the cusps $\frac{1}{2}$ and $0$ are
\begin{align*}
F(z)|_{2}\left(\begin{array}{cc}1 & 0 \\2 & 1\end{array}\right) &= \theta_{0}^{4}(z) \in \Z_{(\ell)} [\![q]\!]\\
\theta_{0}^{2}(z)|_{2}\left(\begin{array}{cc}1 & 0 \\2 & 1\end{array}\right) &= \psi^{4}(z) \in \Z_{(\ell)} [\![q^{1/2}]\!],
\end{align*}
and
\begin{align*}
F(z)|_{2}\left(\begin{array}{cc}0 &-1 \\1 & 0\end{array}\right) &= -\frac{1}{64} \frac{\eta^{8}(z/4)}{\eta^{4}(z/2)} \in \Z_{(\ell)} [\![q^{1/4}]\!]\\
\theta_{0}^{2}(z)|_{2}\left(\begin{array}{cc}0 &-1 \\1 & 0\end{array}\right) &= -\frac{i}{2} \theta_{0}^{2}(z/4) \in i \Z_{(\ell)} [\![q^{1/4}]\!],
\end{align*}
where $\psi(z)=\sum_{j=0}^{\infty}q^{(j+1/2)^{2}}$ and $\eta(z):= q^{1/24} \prod_{n=1}^{\infty}(1-q^{n})$.

\begin{rem}\label{weightbig}
Let $f\in M_k$ be non-zero where $k\in\Z$.  Then $f\in M_{k}(\Gamma_{1}(4))=M_k(\Gamma_0(4), \chi_{-1}^k)$ and the valence formula for $\Gamma_0 (4)$ shows that the total number of zeros of $f$ is $(k/12)[\Gamma_0(1) : \Gamma_0(4)] = k/2$.  In particular $\ord_0 f + \ord_{1/2} f + \ord_\infty f \leq k/2$ with equality exactly when $f$ is non-vanishing on the upper half plane.
\end{rem}

If $f\in M_k$, then denote its coefficient-wise reduction modulo $\ell$ by $\overline{f} := f \mod \ell \in \F_\ell [\![ q]\!]$ and the set of all such reduced forms by
\begin{align*}
\overline{M}_k = \left\{ \overline{f} : f\in M_k\right\}.
\end{align*}
For $\overline{f}=\sum_{n=0}^{\infty} a(n)q^{n}\in \overline{M}_{k}$ with $k\in Z$, we define the filtration
\begin{equation*}
\omega(f) = \omega(\overline{f}) := \inf \left\{ k':\overline{f} \in \overline{M}_{k'} \right\}
\end{equation*}
and the order at the infinite cusp
\begin{align*}
\ord_{\infty}(\overline{f}) := \inf \left\{ n : a(n)\not\equiv 0 \mod \ell\right\}.
\end{align*}
For $k\geq 4$ even, let $E_k$ be the weight $k$ normalized Eisenstein series on $\operatorname{SL}_2(\Z)$.  It is well-known that $E_{\ell-1}, E_{\ell+1} \in M_k$, $\overline{E_{\ell-1}} = 1$, and $\overline{E_{\ell+1}} = \overline{E_2}$, where we let $E_2$ be the weight 2 quasi-modular Eisenstein series.  Define the operator
\begin{equation*}
\Theta := \frac{1}{2\pi i} \frac{d}{dz}.
\end{equation*}
Although it does not map modular forms to modular forms, if $f\in M_k$ then $12\Theta f - k E_2 f\in M_{k+2}$.  Along these lines, define
\begin{equation}\label{ris}
R(f) := \left( \Theta f - \frac{k}{12} E_2 f \right) E_{\ell-1} + \frac{k}{12}E_{\ell+1}f \in M_{k +\ell + 1},
\end{equation}
so that $\overline{R(f)} = \overline{\Theta f}$.  The definition of $R(f)$ implicitly depends on the weight of $f$.  We recursively define
\begin{align*}
R_1^f &:= R(f)\\
R_i^f &:= R(R_{i-1}^f)\in M_{k+i(\ell+1)},
\end{align*}
so that
\begin{align}
\overline{R_i^f}=\overline{\Theta^i f}.\label{Ri is Thetai}
\end{align}

Define $U_\ell$ on power series by
\begin{equation*}
\left( \sum a(n)q^n \right) |U_\ell = \sum a(\ell n)q^n.
\end{equation*}
Fermat's Little Theorem easily provides the relation
\begin{equation}\label{uandtheta}
\overline{\left(f|U_\ell\right)^\ell} = \overline{f - \Theta^{\ell-1} f}.
\end{equation}

The action of $\Theta$ on $\overline{M}_k$ is similar to the well-known level 1 case.  Using $\theta_0$ and $F$ as defined in (\ref{Fdefn}-\ref{theta defn}), Tupan~\cite{tup2006a} proves there is a polynomial $A(X,Y)\in \Z_{(\ell)} [X,Y]$ such that $A(\theta_0^4,F)=E_{\ell-1}$, and further provides an explicit structural isomorphism
\begin{equation}
\frac{\F_\ell[X,Y]}{A(X^4,Y) -1} \to \bigoplus_{k=0}^\infty \overline{M}_k.\label{explicit_structural_iso}
\end{equation}
This allows one to prove the following (See, e.g., \cite{ahlchorou2009a} Proposition 2).
\begin{lem}\label{filtlemma}
Let $\ell \geq 5$ be prime and $f,g$ modular forms on $\Gamma_1(4)$ with coefficients in $\Z_{(\ell)}$.
\begin{enumerate}
\item $\omega(\Theta f) \leq \omega(f) + \ell + 1$ with equality if and only if $\omega(f)\not\equiv 0\mod \ell$,
\item If $f$ and $g$ have weights $k_1$ and $k_2$, respectively, and if $f\equiv g\mod \ell$ then $k_1 \equiv k_2 \mod {\ell-1}$, and
\item For $i\geq 0 $, $\omega(f^i)=i\omega(f)$.
\end{enumerate}
\end{lem}
For general $\Gamma_1(N)$, $N\geq 4$, one could appeal to Section 4 of \cite{gro1990a} for an analogous lemma.

One may use the isomorphism (\ref{explicit_structural_iso}) to deduce the direct sum decomposition
\begin{align}\label{direct_sum_decomp}
\overline{\bigoplus_{k=1}^\infty M_k} = \bigoplus_{\alpha \mod {\ell-1}} \overline{M^\alpha}
\end{align}
where $\overline{M^\alpha}=\cup_{k\equiv \alpha \mod{\ell-1}} \overline{M_k}$.

\begin{rem}\label{fallfilt}
When $\omega(f)\equiv 0\mod \ell$ the above lemma implies $\omega(\Theta f) = \omega(f) + \ell + 1 -s (\ell-1)$ with $s\geq 1$.
\end{rem}

For any $f\in M_{k}$, write $f|_{k} \left(\begin{array}{cc}1 & 0 \\2 & 1\end{array}\right) = \sum_{n=0}^{\infty} b(n/2)q^{n/2}$ and $f|_{k}\left(\begin{array}{cc}0 &-1 \\1 & 0\end{array}\right) =i^{k}\sum_{n=0}^{\infty} c(n)q^{n/4}$ and define
\begin{align*}
\ord_{1/2}(\overline{f}) &:= \inf \left\{ n/2 : b(n/2)\not\equiv 0 \mod \ell\right\}\\
\ord_{0}(\overline{f}) &:= \inf \left\{ n : c(n)\not\equiv 0 \mod \ell\right\}.
\end{align*}
It follows that for any of the cusps $s$ we have 
\begin{align}
\ord_{s} (\overline{f}) \geq \ord_{s}(f).\label{ord_overline_f_geq_ord_f}
\end{align}

\begin{rem}\label{reduced order is well defined}
For any cusp $s$, $\ord_{s}(\overline{f})$ is well-defined in the sense that if a power series $\sum a(n)q^{n} \in\F_{\ell}[\![ q]\!]$ is congruent to both $f(z)\in M_{k}$ and $g(z)\in M_{k+m(\ell-1)}$, then by Lemma~\ref{filtlemma}~(2),
\begin{align*}
f(z)E_{\ell-1}^{m}=g(z) +\ell h(z) 
\end{align*}
for some $h(z)\in M_{k+m(\ell-1)}$.  Now 
\begin{align*}
f(z)E_{\ell-1}^{m} |_{k+m(\ell-1)} \left(\begin{array}{cc}1 & 0 \\2 & 1\end{array}\right) &= f(z)|_{k} \left(\begin{array}{cc}1 & 0 \\2 & 1\end{array}\right) E_{\ell-1}^{m}\\
&\equiv f(z)|_{k} \left(\begin{array}{cc}1 & 0 \\2 & 1\end{array}\right) \mod\ell
\end{align*}
and
\begin{align*}
(g(z) +\ell h(z))|_{k+m(\ell-1)} \left(\begin{array}{cc}1 & 0 \\2 & 1\end{array}\right) &= g(z)|_{k+m(\ell-1)}\left(\begin{array}{cc}1 & 0 \\2 & 1\end{array}\right) +\ell h(z)|_{k+m(\ell-1)} \left(\begin{array}{cc}1 & 0 \\2 & 1\end{array}\right)\\
&\equiv g(z)|_{k+m(\ell-1)}\left(\begin{array}{cc}1 & 0 \\2 & 1\end{array}\right)\mod\ell.
\end{align*}
The situation for the cusp $0$ is similar.
\end{rem}

A short computation (for example \cite{sin2009b} Lemma 4.2) shows that
\begin{align}
R(f)|_{k+\ell+1}\gamma = \left( \Theta(f|_k\gamma) - \frac{k}{12} E_2 (f|_k\gamma) \right) E_{\ell-1} + \frac{k}{12}E_{\ell+1} (f|_k\gamma).\label{rgamma}
\end{align}
\begin{lem}\label{reducedordri}
If $f\in M_k$, $k\in\Z$, then for every cusp $s\in \{0, 1/2, \infty\}$ and $i\geq 1$, $\ord_s \left( \overline{R_i^f} \right) \geq \ord_s (f)$.
\end{lem}
\begin{proof}
First recall that for $k\geq 2$, $E_k = 1 + O(q)$.  Hence $\ord_\infty E_k = 0$.
For the cusp $s=\infty$, by Equation (\ref{ris})
\begin{align*}
\ord_\infty (R(f)) &\geq \min\{ \ord_\infty (\Theta f), \ord_\infty (f) +1\}\\
&= \ord_\infty (f).
\end{align*}
For the cusp $s=0$, set $\gamma = \begin{pmatrix} 
  0 & -1\\ 
  1 & 0 
\end{pmatrix}$.  By Equation (\ref{rgamma})
\begin{align*}
\ord_0 (R(f)) &= 4 ~\ord_{\infty} \left( R(f)|_{k+\ell+1} \gamma \right)\\
&\geq 4 \min \{ \ord_\infty ( \Theta (f|_k \gamma ) ), \ord_\infty ( f|_k \gamma ) +1\}\\
&\geq4~ \ord_\infty ( f|_k \gamma)\\
&=\ord_0 (f).
\end{align*}
Similarly $\ord_{1/2} (R(f))\geq\ord_{1/2} (f)$.  For all cusps $s$, iteration yields $\ord_s (R_i^f)\geq\ord_s (f)$.  Equation (\ref{ord_overline_f_geq_ord_f}) gives the conclusion.
\end{proof}

\section{The Tate Cycle}\label{tcbasics}
The following framework follows Jochnowitz \cite{joc1982a}.  Let $f\in M_k$, $k\in \Z$, be such that $\overline{\Theta f}\neq 0$.  Clearly by Fermat's Little Theorem $\overline{\Theta f} = \overline{\Theta^\ell f}$.  The sequence $\overline{\Theta f}, \overline{\Theta^2f}, \dots, \overline{\Theta^\ell f}$ is called the Tate cycle of $f$.  We say that $f$ is in its own Tate cycle if $\overline{f}= \overline{\Theta^{\ell-1} f}$.  By Lemma \ref{filtlemma}, the filtration of the Tate cycle will naturally rise and fall.  Since the increases in filtration are bounded by $\ell +1$ and since the cycle is periodic, the aggregate decreases in filtration are bounded.  In addition, unless $\omega(f)\equiv 0\mod\ell$, we have $\omega(\Theta f) \equiv \omega(f) +1\mod\ell$ and so falls are both predictable and rare.  Call $\Theta^{i} f$ a high point and $\Theta^{i+1} f$ a low point of the Tate cycle when $\omega(\Theta^{i} f)\equiv 0 \mod\ell$.  Then by Remark \ref{fallfilt}, $\omega(\Theta^{i+1} f) = \omega(\Theta^{i} f) + \ell + 1 - s(\ell-1) \equiv 1+s\mod\ell$ with $s \geq 1$.

\begin{lem}\label{tate_cycle_basics_lemma}
Let $f\in M_k$ with $\omega(f) = k = A\ell + B$, where $0\leq B \leq \ell-1$.  Suppose $\overline{\Theta f}\neq 0$.
\begin{enumerate}
 \item If $\omega(f)\equiv 1 \mod\ell$ then $f$ is not in its Tate cycle.
\item The low point of a Tate cycle has filtration $2 \mod\ell$ if and only if the Tate cycle has exactly one drop.
\item The Tate cycle of $f$ has either one or two low points.
\item We never have $\omega(\Theta^{j+1} f) = \omega(\Theta^{j} f) +2$ with $j\geq 1$.  That is, the filtration never rises by two inside a Tate cycle.
\item Assume $\overline{f}=\overline{\Theta^{\ell-1} f}$ is in its own Tate cycle, that $\overline{f}$ is a low point, and that there are two low points.  Let $\Theta^{i_1} f$ and $\Theta^{i_2}f$ be the high points with $1\leq i_1 < i_2 = \ell-2$.  Let $s_1$ and $s_2$ be the sizes of the falls as in Remark \ref{fallfilt}.  Then  $i_1 = \ell - B, i_2 = \ell - 2, s_1=\ell -B +2, s_2 = B-1$ and the filtrations of the high and low points are
\begin{eqnarray*}
\omega(\Theta^{i_1}f) &=& \omega(f) +i_1(\ell+1)\\
\omega(\Theta^{i_1 +1}f) &=& \omega(f) +(i_1+1)(\ell+1)-s_1(\ell-1) =\omega(f)+\ell+3-2B\\
\omega(\Theta^{i_2}f) &=& \omega(f) +i_2(\ell+1)-s_1(\ell-1)\\
\omega(\Theta^{\ell-1} f) = \omega(\Theta^{i_2 +1}f) &=& \omega(f) +(i_2+1)(\ell+1)-(s_1+s_2)(\ell-1)= \omega(f).
\end{eqnarray*}
\end{enumerate}
\end{lem}

\begin{proof}
(1) If $\omega(f)\equiv 1 \mod\ell$, then by Lemma \ref{filtlemma} (1), for $0 \leq i \leq \ell-1$ we have $\omega(\Theta^i f)= \omega(f) +i(\ell+1) \equiv 1+i\mod\ell$.  That is, $\omega(f) < \omega(\Theta f) < \cdots < \omega(\Theta^{\ell-1})$ and so $\overline{f} \neq \overline{\Theta^{\ell-1}f}$.

(2) If the low point of a Tate cycle, $\overline{g}$, has $\omega(g) \equiv 2\mod\ell$, then by Lemma \ref{filtlemma} (1), for $0 \leq i \leq \ell-2$ we have $\omega(\Theta^i g)= \omega(g) +i(\ell+1) \equiv 2+i\mod\ell$.  Then $\overline{g},\dots, \overline{\Theta^{\ell-2} g}$ are $\ell-1$ distinct elements of the cycle.  Hence, the next iteration must be $\overline{\Theta^{\ell-1}g}=\overline{g}$.  Conversely, if there is only one drop, then there must be $\ell-2$ increases in the filtration before the single fall.  Then by Lemma \ref{filtlemma} the low point must have filtration $2\mod\ell$.  Note that in the case of a single drop in filtration, the $s$ in Remark \ref{fallfilt} is $s=\ell +1$.

(3) If $g$ is a low point of the Tate cycle of $f$ and the high points are labelled $\Theta^{i_1} g, \dots \Theta^{i_t} g$ and $t\geq 2$, then since $\overline{g}=\overline{\Theta^{\ell-1}g}$ is a low point, $i_t=\ell-2$.  In order to examine the change in filtration between consecutive high points, it is convenient to let $i_{t+1}= i_1+\ell-1$.  By Remark~\ref{fallfilt} and part (2) above, for each $1\leq j\leq t$ we have $s_j\geq 2$ such that $\omega(\Theta^{i_j+1} g) = \omega(\Theta^{i_j} g) + \ell + 1 - s_j(\ell-1)\equiv 1 + s_j \mod \ell$.  Then $i_{j+1}-i_j \equiv - s_j\mod\ell$.  Considering the full Tate cycle
\begin{align*}
\omega(g)=\omega(\Theta^{\ell-1}g) &= \omega(g) + (\ell-1)(\ell+1) - \sum_{j=1}^t s_j (\ell-1)
\end{align*}
and so $\sum s_j = \ell+1$.  Since $t\geq 2$, for $1\leq j\leq t$ we deduce $i_{j+1}-i_j = \ell- s_j$ from the previous congruence.  Now $\ell-1 = \sum_{j=1}^{t} (i_{j+1} - i_j) = t\ell - \sum s_j = t\ell - (\ell +1)$ which implies $t=2$.

(4) By Lemma \ref{filtlemma} (1), $\omega(\Theta^j f) = \omega(\Theta^{j+1} f) +2$ implies $\omega(\Theta^j f)\equiv 0\mod\ell$.  Then $\omega(\Theta^{j+1} f) \equiv 2 \mod\ell$.  As in the proof of part (2), the filtration increases for $\ell-2$ more times before falling.  Hence $\omega(\Theta^{j+1+\ell-2} f) > \omega(\Theta^j f)$ and so $\overline{\Theta^j f} \neq \overline{\Theta^{j+\ell -1} f}$ which implies $\Theta^j f$ is not in its Tate cycle and hence $j=0$. 

(5) This part simply collects what we already know.  Since $\omega(f) \equiv B\mod\ell$, by Lemma \ref{filtlemma}, $i_1 = \ell - B$.  The values of $s_j$ are found by recalling $s_1 + s_2 = \ell+1$ and $i_2-i_1 = \ell- s_1$ from the proof of part (3).  Remark \ref{fallfilt} provides the filtrations.
\end{proof}

\begin{rem}\label{when_do_we_have_lowest_low_point}
By part (5) of the above lemma, if $f$ is a low point of its Tate cycle, it will be the lowest of two low points exactly when $B\geq 3$ and
\begin{eqnarray*}
\omega(f) +(i_1+1)(\ell+1)-s_1(\ell-1) > \omega(f),
\end{eqnarray*}
or equivalently when $3 \leq B< \frac{\ell+3}{2}$.  If $f$ is a low point with $B=\frac{\ell+3}{2}$ then $s_1=s_2=\frac{\ell+1}{2}$ and $\omega(f)=\omega(\Theta^{\frac{\ell-1}{2}}f)$ are both low points.  Conversely, if $f$ is one of two low points, each with the same filtration, then $B=\frac{\ell+3}{2}$.
\end{rem}

\section{Congruences and equivalent properties}\label{KOsection}

We generalize the work of Kiming and Olsson \cite{kimols1992a} to modular forms on $M_k(\Gamma_1(4))$. 
\begin{definition}
A power series $f=\sum b(n) q^n\in \Z_{(\ell)}[\![q]\!]$ has a congruence at $a \mod \ell$ if for all integers $n$,
\begin{equation*}
b(\ell n + a) \equiv 0\mod \ell.
\end{equation*}
\end{definition}

\begin{lem}\label{KOlemma}
Let $f=\sum b(n)q^n$ and $g=\left(\sum c(n)q^{n}\right)^{\ell}\equiv \sum c(n)q^{\ell n} \not\equiv 0\mod\ell$.  The series $f$ has a congruence at $a\mod\ell$ if and only if $fg$ has a congruence at $a\mod\ell$.
\end{lem}

\begin{proof}
Write $fg=\sum d(n) q^n$ where $d(n) = \sum b(n-i\ell)c(i)$.  The result follows.
\end{proof}

\begin{rem}\label{things_equiv_to_congs}
By Equation (\ref{uandtheta}), $f$ has a congruence at $0\mod \ell$ if and only if
\begin{equation*}
f|U_\ell \equiv 0 \mod \ell \iff (f|U_\ell)^\ell \equiv 0 \mod \ell \iff f\equiv \Theta^{\ell-1}f \mod \ell. 
\end{equation*}
Furthermore, $f$ has a congruence at $a\mod \ell$ if and only if $q^{-a}f$ has a congruence at $0\mod \ell$.  Equivalently, $f$ has a congruence at $a\mod\ell$ if and only if
\begin{equation*}
(q^{-a}f)|U_\ell \equiv 0 \mod \ell \iff  q^{-a}f\equiv \Theta^{\ell-1}\left(q^{-a}f\right) \mod \ell.
\end{equation*}
\end{rem}

The following wonderful lemma comes from the proof of Proposition~3 of Kiming and Olsson~\cite{kimols1992a}.

\begin{lem}\label{unexlem}
A modular form $f\in M_k$ with $\overline{\Theta f}\neq 0$ has a congruence at $a\not\equiv 0\mod \ell$ if and only if  $\Theta^\frac{\ell +1}{2}f\equiv -\leg{a}{\ell}\Theta f \mod \ell$.
\end{lem}
\begin{proof}
Since $\Theta$ satisfies the product rule,
\begin{eqnarray*}
\Theta^{\ell-1}\left(q^{-a}f\right) &\equiv& \sum_{i=0}^{\ell -1} \binom{\ell-1}{i}(-a)^{\ell-1-i}q^{-a}\Theta^i f \mod \ell\\
&\equiv& \sum_{i=0}^{\ell -1} a^{\ell-1-i}q^{-a}\Theta^i f \mod \ell\\
&\equiv& a^{\ell-1}q^{-a}f + \sum_{i=1}^{\ell -1} a^{\ell-1-i}q^{-a}\Theta^i f \mod \ell.
\end{eqnarray*}
A congruence at $a\not\equiv 0\mod \ell$ is thus equivalent to $0\equiv \sum_{i=1}^{\ell -1} a^{\ell-1-i}q^{-a}\Theta^i f \mod \ell$, and hence to $0\equiv \sum_{i=1}^{\ell -1} a^{\ell-1-i}\Theta^i f \mod \ell$.  By Lemma \ref{filtlemma}, for $1\leq i \leq \frac{\ell-1}{2}$ we have
\begin{eqnarray*}
\omega(\Theta^i f) \equiv \omega(\Theta^{i +\frac{\ell-1}{2}} f) \equiv \omega(f) +2i \mod{\ell-1}.
\end{eqnarray*}
By Lemma \ref{filtlemma} (2) and Equation (\ref{direct_sum_decomp}), the only way for the given sum to be zero is if for all $1\leq i\leq \frac{\ell-1}{2}$,
\begin{equation*}
a^{\ell -1-i}\Theta^i f +a^{\ell -1-(i+\frac{\ell-1}{2})}\Theta^{i+ \frac{\ell-1}{2}}f \equiv 0 \mod\ell,
\end{equation*}
which happens if and only if for each $i$
\begin{align*}
\Theta^{i + \frac{\ell-1}{2}}f &\equiv -a^\frac{\ell-1}{2}\Theta^i f \equiv - \leg{a}{\ell}\Theta^i f \mod \ell
\end{align*}
which happens if and only if
\begin{align*}
\Theta^{\frac{\ell+1}{2}}f \equiv - \leg{a}{\ell} \Theta f \mod\ell.
\end{align*}
\end{proof}

\section{Lifting data to characteristic zero}\label{handybasissection}
Recall that we denote $M_k := M_k(\Gamma_1(4))\cap \Z_{(\ell)} [\![q]\!]$, with $k\in\Z$.  Consider the forms
\begin{align*}
E &:= \frac{\eta^{8}(z)}{\eta^4(2z)} \in M_2,\\
F &= \frac{\eta^{8}(4z)}{\eta^4(2z)} = \sum_{n\geq 0} \sigma_1(2n+1)q^{2n+1}\in M_2,\\
\theta_0^2 &= \frac{\eta^{10}(2z)}{\eta^4(z)\eta^4(4z)} = \left(\sum_{n\in \Z} q^{n^2}\right)^2 \in M_1.
\end{align*}
Note that $\ord_0 (E)=1$, $\ord_\infty (F)=1$, $\ord_{1/2} (\theta_0^2) = 1/2$, and that these are the only zeros of these forms.  Since $\dim M_{k} = 1 + \lfloor k/2 \rfloor$, one sees
\begin{align}
M_{2k} &=  \langle E^{k-i}F^i\rangle_{i=0,1,\dots,k}\label{EFbasis}\\
M_{2k+1} &= \theta_0^2\langle E^{k-i}F^i\rangle_{i=0,1,\dots,k},\nonumber
\end{align}
as $\Z_{(\ell)}$-modules, where the basis vectors $E^{k-i}F^i=q^i + \cdots$ have rising orders at $\infty$.  The following modification (partially) arranges for ascending orders at the other cusps as well.  Fix non-negative integers $m_\infty, m_0, m_{1/2}$ such that $m_\infty + m_0 + m_{1/2}\leq k$ and set 
\begin{equation*}
G:=\theta_0^4 = E+16F \in M_2.
\end{equation*}
Define the following submodules of $M_{2k}$ depending on $\overline{m}=(m_\infty, m_0, m_{1/2}, 2k)$.
\begin{eqnarray}
V^{\overline{m}} &:=& \{ f\in M_{2k} | \text{ for all cusps } s, \ord_s f \geq m_s \}\nonumber\\
&=& E^{m_0}F^{m_\infty}G^{m_{1/2}} M_{2(k-m_0 - m_\infty - m_{1/2})}\nonumber\\
&=& \langle E^{k-m_\infty -m_{1/2} - i}F^{m_\infty +i} G^{m_{1/2}} \rangle_{i=0,1,\dots, k-m_0 -m_\infty - m_{1/2}}\nonumber\\
W_{\infty}^{\overline{m}} &:=& \langle E^{k-i}F^i\rangle_{i=0,1,\dots m_\infty -1}\label{VW_basis_defn}\\
W_{0}^{\overline{m}} &:=& \langle E^{i}F^{k-i}\rangle_{i=0,1,\dots m_0 -1}\nonumber\\
W_{1/2}^{\overline{m}} &:=& \langle E^{m_0}F^{k-m_0-i}G^i\rangle_{i=0,1,\dots m_{1/2} -1}\nonumber
\end{eqnarray}
so that each $W_s^{\overline{m}}$ has $m_s$ basis forms, each with distinct order at $s$.  In particular $W_s^{\overline{m}} \subseteq \{f\in M_{2k} |\ord_s f< m_s \}$.  In addition, each form in (\ref{VW_basis_defn}) has a different order at $\infty$.  It follows that (\ref{VW_basis_defn}) has $k$ linearly independent basis vectors and
\begin{eqnarray*}
 M_{2k} = V^{\overline{m}} \oplus W_\infty^{\overline{m}} \oplus W_0^{\overline{m}} \oplus W_{1/2}^{\overline{m}} 
\end{eqnarray*}
as a $\Z_{(\ell)}$-module.  We have the following lifting result.
\begin{prop}\label{rigidityprop}
Let $m_\infty, m_0, m_{1/2}, k$ be non-negative integers satisfying $m_\infty + m_0 + m_{1/2}\leq k$.  Set $\overline{m}=(m_\infty, m_0, m_{1/2}, 2k)$.  Let $V^{\overline{m}}$ and the $W_{s}^{\overline{m}}$ be subspaces of $M_{2k}(\Gamma_1(4))\cap \Z_{(\ell)}[\![q]\!]$ as in (\ref{VW_basis_defn}).

(a) If $f\in M_{2k}$ has $\ord_s (\overline{f}) \geq m_s$ for all cusps $s$, then we can write $f=g+\ell h$, where $g\in V^{\overline{m}}$ and $h\in W_0^{\overline{m}} \oplus W_\infty^{\overline{m}} \oplus W_{1/2}^{\overline{m}}$.

(b) If $f'\in M_{2k+1}$ has $\ord_s (\overline{f'}) \geq m_s$ for all cusps $s$, then $f' = \theta_0^2 f$ for some $f\in M_{2k}$ with $\ord_s(\overline{f}) \geq m_s$ for all cusps $s$.  (Recall $m_{1/2}\in \Z$.)  There are $g\in V^{\overline{m}}$ and $h\in W_0^{\overline{m}} \oplus W_\infty^{\overline{m}} \oplus W_{1/2}^{\overline{m}}$ such that $f'=\theta_0^2 g+\ell \theta_0^2 h$.\end{prop}
\begin{proof}
Write $f=g+h_\infty + h_0 + h_{1/2}$, where $g\in V^{\overline{m}}$ and $h_s\in W_s^{\overline{m}}$.  We show each $\overline{h_s}=0$. (It is important to do this in the correct order.)  If $W_\infty^{\overline{m}} \neq \emptyset$, then let $h_\infty=\sum_{i=0}^{m_\infty -1} a_i E^{k-i}F^i$, with $a_i\in\Z_{(\ell)}$.  If any $a_i\not\equiv 0\mod \ell$, then let $t$ be the least such $i$.  In this case, $h_\infty \equiv a_t q^t + \cdots \mod\ell$ has order $t$.  By construction $V^{\overline{m}}\oplus W_0^{\overline{m}} \oplus W_{1/2}^{\overline{m}}$ only contains forms of order at least $m_\infty$ at the infinite cusp.  Hence
\begin{eqnarray*}
m_\infty \leq \ord_\infty \left( \overline{f} \right)=\ord_\infty \left( \overline{h_\infty} \right)=t< m_\infty,
\end{eqnarray*}
a contradiction.  Thus $\overline{h_\infty}=0$.

Now consider $h_0=\sum_{i=0}^{m_0 -1} b_i E^iF^{k-i}$ with $b_i\in \Z_{(\ell)}$.  If any $b_i\not\equiv 0\mod \ell$, then let $t$ be the least such $i$.  Then $\ord_0 (h_0) = t \leq m_0 -1$.  Since $V^{\overline{m}}\oplus W_{1/2}^{\overline{m}}$ only contains forms with order at least $m_0$ at zero and since $\overline{h_\infty} = 0$,
\begin{align*}
m_0 \leq \ord_0 \left( \overline{f} \right)=\ord_0 \left( \overline{h_0} \right)=t< m_0,
\end{align*}
a contradiction.  Thus $\overline{h_0} = 0$.  An analogous argument shows that if $\overline{h_{1/2}}\neq 0$, then
\begin{align*}
m_{1/2} \leq \ord_{1/2} \left( \overline{f} \right)=\ord_{1/2} \left( \overline{h_{1/2}} \right)< m_{1/2},
\end{align*}
another contradiction.  For part (b), recall that any $f'\in M_{2k+1}$ must have $\ord_{1/2} f' \in \Z + \frac{1}{2}$ and hence is divisible by $\theta_0^2$.  Apply part (a) to $f=f'/\theta_0^2 \in M_{2k}$.
\end{proof}

We have the following Sturm-syle result.

\begin{cor}\label{sturm_cor}
(a) Let $f\in M_{2k}$ and $\ord_0 \left(\overline{f}\right) + \ord_\infty \left(\overline{f}\right) + \ord_{1/2} \left(\overline{f}\right) > k$.  Then for all cusps $s$, $\ord_s \left(\overline{f}\right) = + \infty$ and $\overline{f} =0$.

(b) Let $f\in M_{2k+1}$ and $\ord_0 \left(\overline{f}\right) + \ord_\infty \left(\overline{f}\right) + \ord_{1/2} \left(\overline{f}\right) > k +1/2$.  Then for all cusps $s$, $\ord_s \left(\overline{f}\right) = + \infty$ and $\overline{f} =0$.
\end{cor}

\begin{proof}
(a) Suppose $\overline{f} \neq 0$.  For each cusp $s$, choose integers $0\leq m_s \leq \ord_s\left(\overline{f}\right)$ such that $m_0 + m_\infty + m_{1/2} = k$.  Set $\overline{m} = (m_\infty, m_0, m_{1/2}, 2k)$ and apply Proposition \ref{rigidityprop}.  Write $f=g +\ell h$, with $g\in V^{\overline{m}}$ and $h\in W_0^{\overline{m}} \oplus W_\infty^{\overline{m}} \oplus W_{1/2}^{\overline{m}}$.  For the parameters in $\overline{m}$, $\dim V^{\overline{m}} = 1$.  Therefore, $g = c E^{m_0} F^{m_\infty} G^{m_{1/2}}\in M_{2k}, c\in \Z_{(\ell)}$.  We now have a contradiction since for any cusp $s$, $\ord_{s}(\overline{f}) = \ord_{s}(\overline{g})=m_{s}$, contrary to our assumption that $\sum \ord_{s}(\overline{f}) > k$.

(b) Apply part (a) to $f/\theta_0^2\in M_{2k}$.
\end{proof}

In the next section we use the following proposition to lift a low point of a Tate cycle --- a $\mod\ell$ object --- to a characteristic zero modular form with high orders of vanishing at the cusps. 

\begin{prop}\label{narrowingprop} Let $k'$ and $i$ be positive integers.

(a) Given $f\in M_{2k'}$, let $2k=\omega(\Theta^i f)$ and $m_s=\ord_s f$ for each cusp $s$.  Set $\overline{m} = (m_\infty, m_0, m_{1/2}, 2k)$. Then there is $g\in V^{\overline{m}}$ such that $\overline{\Theta^i f}=\overline{g}$.

(b) Given $f\in M_{2k'+1}$, let $2k+1=\omega(\Theta^i f)$ and $m_s=\lfloor \ord_s f\rfloor$ for each cusp $s$.  Set $\overline{m} = (m_\infty, m_0, m_{1/2}, 2k)$. Then there is $g\in V^{\overline{m}}$ such that $\overline{\Theta^i f}=\overline{\theta_0^2 g}$.
\end{prop}
\begin{proof}
Lemma \ref{reducedordri} implies for each cusp $s$, $\ord_s \left( \overline{R_i^f}\right) \geq \ord_s(f) \geq m_s$.
In the even weight case, apply Proposition \ref{rigidityprop} (a) to deduce $\Theta^i f\equiv R_i^f \equiv g \mod\ell$ for some $g\in V^{\overline{m}}$.  In the odd weight case use Proposition \ref{rigidityprop} (b).
\end{proof}

\section{Congruences in forms which vanish only at the cusps}\label{voacsection}

This section considers modular forms which vanish only at the cusps.  This condition implies a lot about the Tate cycle.  To begin with, if $f\in M_k$, $\overline{\Theta f}\neq 0$, and $f$ vanishes only at the cusps but is not congruent to a cusp form, then $\overline{f|U_\ell}\neq 0$.  This follows from the more general proposition below:

\begin{prop}\label{cuspsonly}
If $0\neq f\in M_k$, $k\in \Z$, and for some cusp $s$, $\ord_s(\overline{f}) \equiv 0\mod\ell$, then $\overline{f|U_\ell}\neq 0$.
\end{prop}
\begin{proof}
Let $\gamma= \left(\begin{array}{cc}1 & 0 \\0 & 1\end{array}\right), \left(\begin{array}{cc}0 & -1 \\1 & 0\end{array}\right)$ or $\left(\begin{array}{cc}1 & 0 \\2 & 1\end{array}\right)$ depending on whether $s=\infty, 0$ or $1/2$, respectively.  Set $c=4$ if $s=0$ and $c=1$ otherwise.  (Thus $c$ is the width of the cusp $s$.)  By examining the orders of the summands in Equation (\ref{rgamma}),

\begin{align*}
\ord_s \left( \overline{R_1^f}\right) &= c ~\ord_\infty \left( \overline{R_1^f|_{k+\ell+1} \gamma}\right) \geq 1 + \ord_s \overline{f}.
\end{align*}

By the proof of Lemma \ref{reducedordri}, $\ord_s \left(\overline{R_{\ell-1}^f}\right)\geq \ord_s \left( \overline{R_1^f}\right) \geq 1 + \ord_s \overline{f}$.  Thus by Remark~\ref{reduced order is well defined} it is impossible for $\overline{R_{\ell-1}^f} = \overline{f}$.  That is,  $\overline{(f|U_\ell)^\ell} = \overline{f} - \overline{\Theta^{\ell-1} f} \neq 0$.
\end{proof}

\begin{prop}\label{voacprop}
Suppose $f\in M_k$, $k\in \Z$, vanishes only at the cusps and $\overline{\Theta f}\neq 0$.  Then for $i\geq 0$, $\omega(\Theta^i f) \geq \omega(f)=k$.  In particular, if $f$ is a member of its own Tate Cycle, then $f$ is a low point.  If $f$ is not a member of its own Tate Cycle, then $\Theta f$ is a low point.
\end{prop}
\begin{proof}
Since $f\in M_k$, obviously $\omega(f)\leq k$.  By Remark \ref{weightbig} and Corollary \ref{sturm_cor}, $\omega(f) \geq k$ and equality follows.  For any $i\geq 1$ and for all cusps $s$, by Lemma \ref{reducedordri}, $\ord_s\left( \overline{R_i^f}\right) \geq \ord_s (f)$.  Hence $\ord_0 \left(\overline{R_i^f}\right) + \ord_\infty \left(\overline{R_i^f}\right) + \ord_{1/2} \left(\overline{R_i^f}\right) \geq k/2$.  By Corollary \ref{sturm_cor} we must have $\omega(\Theta^i f) \geq k$.

Suppose $f$ is not a member of its own Tate cycle and, for the sake of contradiction, that $\overline{\Theta f}= \overline{\Theta^\ell f}$ is not a low point.  By the first assertion, there are two cases:  either $\omega(f) \equiv 0 \mod \ell$ and $\omega(\Theta f) = \omega(f) + 2 \equiv 2 \mod \ell$, or $\omega(\Theta f) = \omega(f)  + \ell + 1$.  In the former case, by Lemma \ref{tate_cycle_basics_lemma} (2) the Tate cycle has a single low point with filtration $2\mod \ell$ and it must then be $\Theta f$.  In the latter case, $k+\ell+1 =\omega(\Theta f) = \omega(\Theta^\ell f) = \omega(\Theta^{\ell-1} f)+\ell + 1$.  In particular $\omega(\Theta^{\ell-1} f) = k$.  However in this case $\dim V^{\overline{m}} =1$.  Therefore $\Theta^{\ell-1} f$ is a constant multiple of $f$ which contradicts $f$ not being in its Tate cycle (since $\Theta$ commutes with scalar multiplication).
\end{proof}


The following two corollaries show the differences between congruences at $a\not\equiv 0 \mod\ell$ and at $0\mod\ell$.  

\begin{cor}\label{unexcor}
Suppose $f\in M_k$, $k\in \Z$, vanishes only at the cusps, $\overline{\Theta f} \neq 0$, and $\omega(f)= A\ell +B$, $0\leq B < \ell$.  If $f$ has a congruence at $a\not\equiv 0\mod\ell$, then either
\begin{enumerate}
 \item $B=\frac{\ell+1}{2}$ and $f$ does not have a congruence at $0\mod\ell$, or
\item $B=\frac{\ell+3}{2}$ and $f$ does have a congruence at $0\mod\ell$.
\end{enumerate}
\end{cor}
\begin{proof}
If $f$ does not have a congruence at $0\mod\ell$, then by Remark \ref{things_equiv_to_congs}, $f$ is not a member of its Tate cycle.  If $B\neq 0$, then by Proposition~\ref{voacprop}, $\Theta f$ is a low point and $\omega(\Theta f) \equiv B+1\mod \ell$.  By Lemma \ref{unexlem},
\begin{align*}
\Theta^{\frac{\ell+1}{2}} f \equiv \pm \Theta f\mod\ell
\end{align*}
and $\omega( \Theta^{\frac{\ell+1}{2}} f ) = \omega(\Theta f) \equiv B+1\mod\ell$.  Now we have the high points
\begin{align*}
\Theta^{\ell-1} f \equiv \pm \Theta^{\frac{\ell-1}{2}} f\mod\ell 
\end{align*}
and so $\Theta^{\frac{\ell+1}{2}} f$ is also a low point.  Since it has the same filtration as the other low point $\Theta f$, by Remark \ref{when_do_we_have_lowest_low_point}, $B+1 \equiv \frac{\ell+3}{2} \mod \ell$.  From the restrictions on $B$, $B=\frac{\ell+1}{2}$.

  If $B=0$, then $\omega(\Theta f) = \omega(f) + \ell+1 - s(\ell-1)$ for $s\geq 1$.  But by Proposition~\ref{voacprop} we deduce $s=1$ and $\omega(\Theta f)\equiv 2\mod\ell$ is a low point.  Hence by Lemma~\ref{tate_cycle_basics_lemma} the filtration has one low point.  This contradicts Lemma~\ref{unexlem} which implies $\omega( \Theta^{\frac{\ell+1}{2}} f ) = \omega(\Theta f)$.

Similarly, if $f$ does have a congruence at $0\mod\ell$, it is a low point of its Tate cycle by Proposition~\ref{voacprop}.  Remark \ref{when_do_we_have_lowest_low_point} and Lemma \ref{unexlem} show there are two equally low low points and $B=\frac{\ell+3}{2}$.
\end{proof}

\begin{cor}\label{excor}
Suppose $f\in M_k$, $k\in \Z$, vanishes only at the cusps, $\overline{\Theta f} \neq \overline{0}$, and $\omega(f)= A\ell +B$ where $0\leq B\leq \ell-1$.  If $B \geq \frac{\ell +5}{2}$, then $\overline{f|U_\ell} \neq \overline{0}$.
\end{cor}

\begin{proof}
If $\overline{f|U_\ell} = 0$, then $f$ is a member of its Tate cycle.  Proposition~\ref{voacprop} implies $f$ is the lowest low point of its cycle, but Remark \ref{when_do_we_have_lowest_low_point} shows that the lowest low point must have $1\leq B\leq \frac{\ell+3}{2}$.
\end{proof}

The following two corollaries eliminate the chance for Ramanujan congruences at all but finitely many primes $\ell$ in half-integral weight forms vanishing only at the cusps, and in the inverses of integral-weight forms vanishing only at the cusps, respectively.

\begin{cor}\label{upper_half_of_range_half_wt_forms_cor}
 Let $f\in M_{\lambda + 1/2}$, $\lambda \in\N$, vanish only at the cusps.  If $\lambda \geq 1$, then $f$ has no congruences for $\ell>2\lambda + 1$.  If $\lambda =0$, then $f$ is a scalar multiple of $\theta_0 = \sum q^{n^2}$ and clearly has congruences at $a\mod\ell$ where $\leg{a}{\ell}=-1$.
\end{cor}
\begin{proof}
In the case $\lambda\geq 2$, by Lemma \ref{KOlemma} it suffices to show $f^{\ell+1}\in M_{(\lambda + 1/2)(\ell+1)}$ has no congruences.  Since $f$ vanishes only at the cusps, so does $f^{\ell+1}$.  Now $\omega(f^{\ell+1}) \equiv \left(\frac{\ell +1}{2}\right)(2\lambda +1) \equiv \frac{\ell + 2\lambda+1}{2} \mod \ell$.  Take $B=\frac{\ell+2\lambda +1}{2}$ in Corollaries \ref{unexcor} and \ref{excor}.

If $\lambda=0$ or $1$, then $f$ is not a cusp form and Proposition \ref{cuspsonly} precludes congruences at $0\mod\ell$.  By Corollary \ref{unexcor}, in the subcase $\lambda =1$ there are no congruences at all.  The subcase $\lambda=0$ is obvious.
\end{proof}

\begin{cor}\label{inverses_of_integral_wt_forms_cor}
If $f\in M_k$, $k\in \Z$, vanishes only at the cusps, then $f^{-1}$ has no congruences for $\ell > 2k+3$.
\end{cor}

\begin{proof}
By Lemma \ref{KOlemma}, the power series $f^{-1}$ has the same congruences as $f^{\ell-1}\in M_{k(\ell-1)}$.  Since $f$ vanishes only at the cusps, so does $f^{\ell-1}$.  Now $\omega(f^{\ell-1}) = k(\ell-1) \equiv \ell-k \mod \ell$ and $\frac{\ell+3}{2} < \ell -k < \ell$.  Take $B=\ell-k$ in Corollaries \ref{unexcor} and \ref{excor}.
\end{proof}

The congruences of the inverse of a half-integral weight modular form are a bit trickier to find, but will always yield to an extension of the Ahlgren-Boylan technique which we illustrate in the following section.

\section{Proof of Theorem \ref{algorithm_thm}}\label{alg_section}
The case of inverses of integral-weight modular forms is covered by Corollary~\ref{inverses_of_integral_wt_forms_cor}.  Thus, let $f\in M_{k/2}$ with $k\geq 3$ odd and vanishing only at the cusps.  Such $f$ must be of the form
\begin{align*}
f=cE^{m_0} F^{m_\infty} \theta_0^{4m_{1/2}}
\end{align*}
where $c\in \Z_{(\ell)}$, $m_{0}, m_{\infty}\in \Z_{\geq 0}$, $m_{1/2}\in \frac{1}{4}\Z_{\geq 0}$, $\ord_{s}f = m_{s}$ and $m_0 + m_\infty + m_{1/2} = k/4$.  Without loss of generality, assume $c=1$.  We provide a method to find all of the finitely many possible primes $\ell\geq 5$ for which there may be a Ramanujan congruence of the sequence $f^{-1}\in \Z_{(\ell)}[\![ q ]\!]$.  Since modular forms on $\Gamma_1(4)$ are completely determined by their first few coefficients, it is always a finite computation to check if any particular prime $\ell$ has Ramanuajan-type congruences.  In this section we eliminate all large $\ell$.  In fact, we assume $\ell>(k+1)(k+3)$.

By Lemma \ref{KOlemma} it suffices to find the congruences for $f^{\ell - 1} \in M_{k(\ell-1)/2}$.  Since $f$ vanishes only at the cusps, the same is true for $f^{\ell - 1}$.  By Lemma \ref{filtlemma} and Proposition~\ref{voacprop}, $\omega(f^{\ell - 1}) = k\frac{\ell-1}{2} \equiv \frac{\ell - k}{2} \mod \ell$.

Let us dispense with the case when $\overline{\Theta f^{\ell-1}}=0$.  Compute
\begin{align*}
\Theta f^{\ell-1} &\equiv \Theta E^{m_{0}(\ell-1)}F^{m_{\infty}(\ell-1)}\theta_{0}^{4m_{1/2}(\ell-1)}\mod\ell\\
&\equiv -m_{\infty}q^{m_{\infty}(\ell-1)} -(8m_{1/2}-8{m_{0}})(1-m_{\infty})q^{m_{\infty}(\ell-1) +1}\\
& ~+ (32m_{0}^{2} + 8m_{0} - 64m_{0}m_{1/2} + 8m_{1/2}+32m_{1/2}^{2})(2-m_{\infty}) q^{m_{\infty}(\ell-1)+2} +\cdots \mod\ell
\end{align*}
If $\overline{\Theta f^{\ell-1}}=0$, then $\ell$ divides $m_{\infty}< k<\ell$.  Thus $m_{\infty}=0$.  From the coefficient of $q^{m_{\infty}(\ell-1) +1}$ above we deduce $m_{1/2}\equiv m_{0}\mod\ell$.  Hence $\ell$ divides $m_{1/2}-m_{0} \leq k < \ell$ and so we deduce $m_{1/2}=m_{0}$.  From the coefficient of $q^{m_{\infty}(\ell-1)+2}$ we now conclude $m_{0}\equiv 0\mod \ell$ and in fact $m_{0}=0$.  In particular, $f=1$ contrary to the choice of $f$.  Therefore, $\overline{\Theta f^{\ell-1}}\neq 0$.

Suppose $f^{\ell-1}$ has a congruence at $a\not\equiv 0\mod\ell$.  Then by Corollary \ref{unexcor}, $\frac{\ell - k}{2} \equiv \frac{\ell +1}{2}$ or $\frac{\ell +3}{2} \mod\ell$.  This can only happen for the finitely many prime divisors of $(k+1)(k+3)$.

Hence, if $\ell > (k+1)(k+3)$ then it is only possible to have a congruence at $0\mod\ell$.  Suppose there is such a congruence.  That is, $f^{\ell-1}$ is a member of its own Tate cycle and by Proposition \ref{voacprop}, $f^{\ell-1}$ is a low point.  Since $\ell > k$, when we write $\omega(f^{\ell - 1}) = A\ell +B$, we can take $0< \frac{\ell-k}{2} = B < \ell$.  By Lemma~\ref{tate_cycle_basics_lemma} (5), the other low point has filtration 
\begin{equation*}
\omega(\Theta^{\frac{\ell+k+2}{2}} f^{\ell-1}) = \omega(f^{\ell-1}) + k+3 = k\left(\frac{\ell-1}{2}\right) + k + 3.
\end{equation*}
By Proposition \ref{narrowingprop}, there is some $g\in M_{k\left(\frac{\ell-1}{2}\right) + k+3}$ such that $\overline{\Theta^{\frac{\ell+k+2}{2}} f^{\ell-1}} = \overline{g}$ and for each cusp $s$, $\ord_s g \geq \ord_s f^{\ell-1}$.  In particular, $g/f^{\ell-1} \in M_{k+3}$.  We may use any convenient basis to represent $M_{k+3}$.  For example,
\begin{align}
\Theta^{\frac{\ell+k+2}{2}} f^{\ell-1} &\equiv g \equiv f^{\ell-1} \left( g/f^{\ell-1} \right)\mod\ell\nonumber\\
&\equiv f^{\ell-1} \left( \sum_{i=0}^{(k+3)/2} a_i E^{\frac{k+3}{2} - i}F^i \right) \mod \ell,\label{quasi_equals_true}
\end{align}
where {\it a priori} $a_i \in \Z_{(\ell)}$, but working $\mod \ell$ allows one to take $a_{i}\in\Z$.  By Proposition \ref{cuspsonly}, $m_\infty = \ord_\infty f \geq 1$.  Since the $\Theta$ operator satisfies the product rule,
\begin{align*}
\Theta^{\frac{\ell+k+2}{2}} f^{\ell-1} &\equiv \Theta^{\frac{\ell+k+2}{2}} \left\{ f^{-1} (q^{\ell m_\infty} + O(q^{\ell (m_\infty + 1)})\right\} \mod \ell \\
&\equiv q^{\ell m_\infty} \Theta^{\frac{\ell+k+2}{2}} f^{-1} + O(q^{\ell m_\infty + \ell -m_\infty}) \mod\ell,
\end{align*}
and similarly
\begin{align*}
f^{\ell-1} \equiv q^{\ell m_\infty} f^{-1}+O(q^{\ell m_\infty + \ell -m_\infty}) \mod\ell,
\end{align*}
implying
\begin{align}
\Theta^{\frac{\ell+k+2}{2}} f^{-1} &\equiv f^{-1} \left( \sum_{i=0}^{(k+3)/2} a_i E^{\frac{k+3}{2} - i}F^i \right) + O(q^{\ell-m_\infty}) \mod \ell. \label{eqn_for_general_alg_with_ai}
\end{align}
Invert $f$ as a Laurent series with integer coefficients.  Write $f^{-1}=\sum_{i=-m_\infty}^\infty b_iq^i$.  Noticing that $\Theta^{\frac{\ell-1}{2}}$ acts by twisting each coefficient by the Legendre symbol, we see
\begin{align}
\Theta^{\frac{\ell+k+2}{2}} f^{-1} &\equiv \leg{\cdot}{\ell} \otimes \Theta^{\frac{k+3}{2}} f^{-1} + O(q^{\ell-m_\infty}) \mod \ell\nonumber\\
&\equiv \sum_{i=-m_\infty}^{\ell-m_\infty-1} \leg{i}{\ell} b_i i^{\frac{k+3}{2}} q^i +O(q^{\ell-m_\infty}) \mod \ell. \label{eqn_for_general_alg_with_legendre}
\end{align}
Truncate the series in (\ref{eqn_for_general_alg_with_legendre}) to keep only the first $(k+5)/2$ terms.  We will solve for the integers $a_i$ in (\ref{eqn_for_general_alg_with_ai}), but this requires making choices for the (finitely many) Legendre symbols in this initial segment.  For each tuple of possible choices for these Legendre symbols, solve for the $a_i$ and proceed as follows.

Lemma~\ref{quasi never true} to follow proves there must be some coefficient at which they are not equal, only congruent.  The difference between these two coefficients must be divisible by $\ell$.  (The prime $\ell$ must also satisfy the choices for the Legendre symbol.)

To summarize, for any half-integer weight modular form $f$, we can always complete these calculations to arrive at a finite list of possible primes $\ell$ for which there is a Ramanujan congruence.  After individually checking each of these primes, one will have found all of the Ramanujan congruences.

\begin{lem}\label{quasi never true}
Let $k\geq 3$ be odd and $\ell > k+4$ be prime.  For any non-zero $f\in M_{k(\ell-1)/2}$ and non-zero $g\in M_{k(\ell-1)/2 +k+3}$, $\Theta^{(\ell+k+2)/2}f\neq g$.
\end{lem}
\begin{proof}
We adapt \cite{atkgar2003a} Proposition 3.3 to suit our specific needs.  The quasi-modular form $\Theta^{(\ell+k+2)/2}f$ is of the form
\begin{align*}
\Theta^{(\ell+k+2)/2}f(\tau) = \sum_{j=0}^{\frac{\ell+k+2}{2}} f_{j}(\tau) E_{2}^{j}(\tau),
\end{align*}
where $f_{j}\in M_{k(\ell-1)/2 +\ell+k+2-2j}$.  Assume $g(\tau)=\sum f_{j}(\tau)E_{2}^{j}(\tau)$ and apply $\tau \mapsto \frac{\tau}{4\tau+1}$.  Recall $E_{2}(\frac{\tau}{4\tau +1})= (4\tau +1)^{2}E_{2}(\tau) -\frac{24i}{\pi}(4\tau+1)$.  Letting $\alpha :=-\frac{24i}{\pi}$,  we have for all $\tau\in \hp$,
\begin{align*}
(4\tau +1)^{k\left( \frac{\ell-1}{2} \right) + k + 3} g(\tau) =
\sum_{j=0}^{\frac{\ell+k+2}{2}} (4\tau +1)^{k\left( \frac{\ell-1}{2} \right) + \ell + k + 2-2j}f_{j}(\tau) \left((4\tau +1)^{2}E_{2}(\tau) + \alpha(4\tau+1)\right)^{j},
\end{align*}
and hence for all $\tau\in\hp$,
\begin{align*}
0=(4\tau +1)^{k+3} g(\tau) - \sum_{m=\frac{\ell+k+2}{2}}^{\ell+k+2} (4\tau +1)^{m}\left(\sum_{\substack{0\leq j\leq \frac{\ell+k+2}{2}\\0\leq s\leq j \\ j=\ell+k+2+s-m}}{j\choose s}\alpha^{j-s} f_{j}(\tau)E_{2}^{s}(\tau)\right).
\end{align*}
Since $g(\tau), f_{j}(\tau)$ and $E_{2}(\tau)$ are all invariant under $\tau\mapsto \tau +1$, the polynomial
\begin{align*}
z^{k+3} g(\tau) - \sum_{m=\frac{\ell+k+2}{2}}^{\ell+k+2} z^{m}\left(\sum_{\substack{0\leq j\leq \frac{\ell+k+2}{2}\\0\leq s\leq j \\ j=\ell+k+2+s-m}}{j\choose s}\alpha^{j-s} f_{j}(\tau)E_{2}^{s}(\tau)\right)
\end{align*}
has infinitely many zeros $z=4\tau +1, 4\tau +5, 4\tau +9, \dots$.  Therefore the coefficients must be zero.  By the assumption $\ell > k+3$, the index $m$ is never $k+2$.  Hence $g(\tau)=0$ contrary to assumption.
\end{proof}

\section{Proofs of Theorems}\label{cooleg}

\begin{proof}[Proof of Theorem \ref{mainthm}]
The cusp forms of least weight on $\Gamma_1(4)$ are generated by
\begin{equation}
f:=\theta_0 FE \in S_{9/2}(\Gamma_1(4)).
\end{equation}
By Lemma \ref{KOlemma} the series $f^{-1}$ will have a congruence at $a\mod\ell$ if and only if $f^{\ell-1}$ has one at $a\mod\ell$.  Since $\omega(f^{\ell-1})=\frac{9}{2}(\ell-1) \equiv \frac{\ell-9}{2} \mod \ell$, by Corollary \ref{unexcor} there can be congruences at $a\not\equiv 0\mod\ell$ only if $\ell = 3$ or $5$.

In the first case, the Sturm bound~\cite{stu1987a} implies that only a short computation is needed to see that $f^2 \equiv -\Theta f^2\mod 3$ and so $f^2 \equiv \Theta^2 f^2\mod 3$.  By Lemma \ref{unexlem}, $f^{-1}$ has congruences at $0\mod 3$ and $1\mod 3$.  In the second case, a finite computation shows that $f^{-1}$ only has congruences for $\ell=5$ at $2\mod 5$ and $3\mod 5$.  Although our machinery does not apply for $\ell=2$, a short calculation shows $f^{-1}$ has a congruence at $0\mod 2$.  An inspection of the coefficients of $q^7, q^{13}$ and $q^{22}$ in $f^{-1}$ shows there are no congruences for $\ell=7,11,13$.  We now move on to $\ell \geq 17$.

Suppose $f^{\ell-1}$ has a congruence at $0\mod\ell$.  The rest of this proof follows Section~\ref{alg_section} so we only provide the explicit calculations.  Now $f^{\ell -1}$ is a low point of its Tate cycle and the other low point is $\omega(\Theta^{\frac{\ell+11}{2}} f^{\ell-1}) = \omega(f^{\ell-1}) + 12$.  Hence
\begin{eqnarray*}
\Theta^{\frac{\ell+11}{2}} f^{\ell-1} \equiv f^{\ell-1} \left( \sum_{i=0}^6 a_i E^{6-i}F^{i} \right) \mod \ell,
\end{eqnarray*}
implying
\begin{eqnarray}
\Theta^{\frac{\ell+11}{2}} f^{-1}\equiv f^{-1} \left( \sum_{i=0}^6 a_i E^{6-i}F^{i} \right) + O(q^{\ell-1}) \mod\ell.\label{flincomb}
\end{eqnarray}
Invert $f$ as a power series with integer coefficients to get
\begin{equation*}
f^{-1} = q^{-1} + 6 + 24q + 80 q^2 + 240 q^3 + 660 q^4 + 1696 q^5 + 4128 q^6 + 9615 q^7 + 21560 q^8 
 + O(q^{9}).
\end{equation*}
We compute
 \begin{align}
\Theta^{\frac{\ell+11}{2}} f^{-1}&\equiv \leg{\cdot}{\ell} \otimes \Theta^6 f^{-1} + O(q^{\ell})\notag\\
&\equiv \leg{-1}{\ell} q^{-1} +24q + \leg{2}{\ell}5120q^2 + \leg{3}{\ell}174960q^3 +2703360q^4\label{legeqn}\\ &~ + \leg{5}{\ell}26500000q^5 +O(q^{6}).\notag
\end{align}
For each of the $2^4$ choices of signs for the Legendre symbols, a computer can easily compute the integers $a_i$ in Equation (\ref{flincomb}).  Comparing the coefficients of $q^6, q^8,$ and $q^{9}$ in Equation (\ref{flincomb}) leads to a contradiction.  For example, suppose $\ell$ satisfies $\leg{-1}{\ell}=\leg{2}{\ell}=-\leg{3}{\ell}=-\leg{5}{\ell}=1$.  One computes that  $a_0=1, a_1=42, a_2=612, a_3=8656, a_4=-76608, a_5=1074912, a_6=-15155584$.  Hence the right side of Equation (\ref{flincomb}) is
\begin{align*}
q^{-1}  &+ 24 q + 5120 q^2  - 174960 q^3  + 2703360 q^4  - 26500000 q^5  - 29891712 q^6  - 911605665 q^7  \\ &- 2744268800 q^8  - 18190442184 q^{9}  - 59662291200 q^{10} - 254616837584 q^{11}  +O(q^{12}),
\end{align*}
whereas the left side may be computed as in Equation (\ref{legeqn}): 
\begin{align*}
q^{-1} &+ 24 q + 5120 q^2  - 174960 q^3  + 2703360 q^4  - 26500000 q^5  - 192595968 q^6  \pm 1131195135 q^7 \\&+ 5651824640 q^8  + 24858684216 q^9  - 98592000000 q^{10} \pm 358875741136 q^{11}+O(q^{12}).
\end{align*}
The $\pm$ come from $\leg{7}{\ell}$ and $\leg{11}{\ell}$.  Since these power series are congruent $\mod\ell$, so are the coefficients of $q^6$ and $q^8$.  But $-29891712\equiv -192595968\mod\ell$ implies $\ell = 2,3,11,13$ or $2963$, while $-2744268800\equiv 5651824640\mod\ell$ implies $\ell =2,5,7$ or $117133$.  Since we've assumed $\ell\geq 17$, we have reached a contradiction.
\end{proof}

\begin{proof}[Proof of Theorem \ref{secondthm}]
Let $g=\theta_0 E^2 F\in S_{13/2}(4)$.  Now $g^{-1}$ will have a congruence at  if and only if $g^{\ell-1}$ does.  Since $\omega(g^{\ell-1}) \equiv \frac{\ell-13}{2}\mod\ell$, Corollary \ref{unexcor} implies there can only be congruences with $a\not\equiv 0\mod\ell$ if $\ell=2$ or $7$.  For $\ell=7$, one checks that $\Theta^4 g^6 \equiv -\Theta g^6$ and by Lemma \ref{unexlem}, $g^6$ and hence $g^{-1}$ have congruences at $1,2,4\mod 7$.

Elementary calculations show no congruences for $0\mod\ell$ when $3\leq \ell\leq 13$.  For $l\geq 17$, if $g^{\ell-1}$ has a congruence at $0\mod\ell$, then it is the lowest low point of its Tate cycle and the other low point is $\omega(\Theta^{\frac{\ell+15}{2}}g^{\ell-1})=\omega(g^{\ell-1}) +16$.  Analogously to Theorem~\ref{mainthm}, $\Theta^{\frac{\ell+15}{2}}g^{-1}\equiv g^{-1}\left( \sum_{i=0}^8 b_i E^{8-i}F^i\right) +O(q^\ell)\mod\ell.$  In the case where $\leg{-1}{\ell} = \leg{2}{\ell} = \leg{3}{\ell} = \leg{5}{\ell}=\leg{7}{\ell}=-1$, solving for the $b_i$ yields $b_0=-1,b_1= -50,b_2= -788,b_3= -175024,b_4= -26446064,b_5= 539142592,b_6= -13397175040,b_7= 271206416128$, and $b_8= -5171059369600$.  Examining the coefficients of $q^8,\dots, q^{12}$ in both sides of the previous equivalence precludes all possible primes $\ell\geq 17$.  The situation for each of the $2^5$ choices for the Legendre symbols is similar.
\end{proof}

The proofs of the remaining theorems all require the same essential tool.  For $d=1,2,4$, although $\eta(dz)\not\in M_{1/2}$,  by \cite{ono2004a} Theorems 1.64 and 1.65 we have $\eta(dz)^{24}\in M_{12}$.  Since $24| \ell^{2}-1$ when $\ell\geq 5$, the strategy in the following proofs is to use Lemma~\ref{KOlemma} to replace occurrences of  $\eta(dz)^{-1}$ with  $\eta(dz)^{\ell^{2}-1}$ and occurrences of  $\eta(dz)$ with  $\eta(dz)^{(\ell^{2}-1)(\ell -1)}$.  This does not change the filtration $\mod\ell$.  Set
\begin{align*}
\delta=\delta_{\ell}:=\frac{\ell^{2}-1}{24}.
\end{align*}

\begin{proof}[Proof of Theorem \ref{overpartthm}]

The overpartition generating function is 
\begin{align*}
\overline{P}(z) = \sum \overline{p}(n) q^n = \prod_{n=1}^\infty \left( \frac{1+q^n}{1-q^n} \right) = \frac{\eta(2z)}{\eta(z)^2}.
\end{align*}
The prime 3 may be checked by direct computation and so we let $\ell\geq 5$ be prime.  By Lemma~\ref{KOlemma}, $\overline{P}(z)$ has a congruence at $a\mod\ell$ if and only if there is a congruence at $a\mod\ell$ for
\begin{align*}
f:=\eta(2z)^{(\ell-1)(\ell^2-1)}\eta(z)^{2(\ell^2-1)} = \left(\eta(2z)^{24(\ell-1)}\eta(z)^{48}\right)^{\frac{\ell^{2}-1}{24}}\in M_{\frac{(\ell-1)(\ell+1)^{2}}{2}}.
\end{align*}
Since $\ord_{\infty}f=\frac{\ell(\ell^{2}-1)}{12}$, by Proposition~\ref{cuspsonly} there is no congruence at $0\mod\ell$.  Since $\omega(f) \equiv \frac{\ell-1}{2} \mod\ell$, by Corollary~\ref{unexcor} there can only be congruences at $a\mod\ell$ if $\frac{\ell-1}{2} \equiv \frac{\ell+1}{2} \mod\ell$ which never happens for $\ell \geq 5$.
\end{proof}

\begin{proof}[Proof of Theorem \ref{crankthm}]
By \cite{chokanlov2009a}, the crank difference generating function is
\begin{equation*}
\sum_{n\geq 0} \left( M_e(n) - M_o(n) \right) q^n = \prod_{n\geq 1} \frac{(1-q^n)^3}{(1-q^{2n})^2} = q^{1/24}\frac{\eta(z)^3}{\eta(2z)^2}.
\end{equation*}
By Lemma~\ref{KOlemma}, when $\ell\geq 5$ this has a congruence at $a\mod\ell$ if and only
\begin{align*}
\left( q^{-1/24} \eta(2z)^{2}\eta(z)^{3(\ell-1)} \right)^{\ell^{2}-1}
\end{align*}
has a congruence at $a\mod\ell$.  This is equivalent to 
\begin{align*}
f:=\eta(z)^{3(\ell-1)(\ell^{2}-1)}\eta(2z)^{2(\ell^{2}-1)} \in M_{\frac{(\ell^2 - 1)(3\ell-1)}{2}}
\end{align*}
having a congruence at $a +\delta\mod\ell$ where we recall $\delta:=\frac{\ell^2 - 1}{24}$.  Since $f$ vanishes only at the cusps, by Proposition~\ref{voacprop}, $\omega(f) = \frac{(\ell^2 - 1)(3\ell-1)}{2} \equiv \frac{\ell+1}{2}\mod\ell.$

The fact that $\omega(f)\equiv \frac{\ell+1}{2} \mod\ell$ is unfortunate.  This is the only time that Corollary~\ref{unexcor} does not rule out congruences at $a\not\equiv 0\mod\ell$.  However, Lemma~\ref{unexlem} guarantees that if there is a congruence at $a\mod\ell$, then in fact there is a congruence at all $b\mod\ell$ such that $\leg{a+\delta}{\ell}=\leg{b+\delta}{\ell}$.

We now apply the method of Section~\ref{alg_section} to find all $\ell$ such that $f$ has a congruence at $0\mod\ell$.  Assume $f|U_{\ell}\equiv 0\mod\ell$.  Then $f$ is a low point of its Tate cycle and by Lemma~\ref{tate_cycle_basics_lemma}, the other low point has filtration $\omega(f) +2$.  Hence by Proposition~\ref{narrowingprop}, $(\Theta^{\frac{\ell+1}{2}}f)/f \in \overline{M}_{2}$.  Since
\begin{align*}
f &\equiv q^{\frac{\ell^{3}-\ell}{8}} \left( q^{\delta} \prod \frac{ (1-q^{n})^{3} }{(1-q^{2n})^{2}}\right) + O\left( q^{\ell + \delta + \frac{\ell^{3}-\ell}{8}} \right)\mod\ell,
\end{align*}
and since $\Theta$ is linear and satisfies the product rule, we obtain
\begin{align*}
\Theta^{\frac{\ell+1}{2}}f \equiv q^{\frac{\ell^{3}-\ell}{8}} \Theta^{\frac{\ell+1}{2}} \left( q^{\delta} \prod \frac{ (1-q^{n})^{3} }{(1-q^{2n})^{2}}\right)+ O\left( q^{\ell + \delta + \frac{\ell^{3}-\ell}{8}} \right)\mod\ell.
\end{align*}
Thus $(\Theta^{\frac{\ell+1}{2}}f)/f$ is congruent to 
\begin{align}
\Theta^{\frac{\ell+1}{2}} ( q^{\delta} &- 3q^{\delta +1} +2q^{\delta+2} + \cdots) \cdot \left(q^{\delta} - 3q^{\delta +1} +2q^{\delta+2} + \cdots\right)^{-1}\mod\ell\nonumber\\
\equiv \delta^{\frac{\ell+1}{2}} &+ \left( 3\delta^{\frac{\ell+1}{2}} -3(\delta+1)^{\frac{\ell+1}{2}}\right) q +\label{crankcong}\\
&  \left( 7\delta^{\frac{\ell+1}{2}} -9 (\delta+1)^{\frac{\ell+1}{2}} + 2(\delta+2)^{\frac{\ell+1}{2}}\right)q^{2}+\cdots \mod\ell.\nonumber
\end{align}
Since this is congruent to a weight two form, and since the basis form $F=q+4q^{3}+\cdots,$ lacks a $q^{2}$ term, we compare the coefficients of $q^{2}$ in $\delta^{\frac{\ell+1}{2}}E=\delta^{\frac{\ell+1}{2}}(1-q+24q^{2}+\cdots)$ and in Equation~(\ref{crankcong}) to deduce $24\delta^{\frac{\ell+1}{2}} \equiv 7\delta^{\frac{\ell+1}{2}} - 9 (\delta+1)^{\frac{\ell+1}{2}} + 2(\delta+2)^{\frac{\ell+1}{2}}\mod\ell$.  Multiplying by $24^{\frac{\ell+1}{2}}$, we find
\begin{align}
-17\leg{-1}{\ell} \equiv -207\leg{23}{\ell} + 94\leg{47}{\ell} \mod\ell.\label{crank leg eq}
\end{align}
That is, $17 \equiv \pm 207 \pm 94\mod\ell$.  The only possible $\ell\geq 5$ are $5,13, 53$ and $71$.  However, only $5$ and $53$ satisfy (\ref{crank leg eq}).  By the equivalences above, $f$ having a congruence at $0\mod\ell$ is equivalent to the crank difference function having a congruence at $a\mod\ell$ with $24a\equiv 1\mod\ell$.  For the primes $5$ and $53$, this means $a=4$ and $42$, respectively.  We have recovered the congruence at $4\mod 5$ of \cite{chokanlov2009a}.  Calculations reveal that the coefficient of $q^{42}$ precludes a congruence at $42\mod{53}$.  
\end{proof}

\begin{proof}[Proof of Theorem \ref{cphi2thm}]
Calculations show there is no congruence for $\ell=3$.  Thus we take $\ell\geq 5$ prime.  Equation~(10.6) of \cite{and1984a} says the generating function of $c\phi_{2}(n)$ is
\begin{equation*}
C\Phi_{2}(z)= \frac{\theta_{0}(z)}{q^{-1/12}\eta(z)^{2}}.
\end{equation*}
Now $C\Phi_{2}$ will have a congruence at $a\mod\ell$ if and only if $\left(q^{-1/12}\theta_{0}(z)^{\ell-1}\eta(z)^{2}\right)^{\ell^{2}-1}$ has a congruence at $a\mod\ell$.  This happens if and only if $f:=\theta_{0} (z)^{(\ell-1)(\ell^{2}-1)} \eta(z)^{2(\ell^{2}-1)}\in M_{(\ell-1)(\ell+1)^{2}/2}$ has a congruence at $a+2\delta \mod\ell$.  Since $f$ vanishes only at the cusps, Proposition~\ref{voacprop} implies $\omega(f) =\frac{(\ell-1)(\ell+1)^{2}}{2} \equiv \frac{\ell-1}{2}\mod\ell$.  By Corollary~\ref{unexcor}, there are no congruences at $a\not\equiv 0\mod\ell$ when $\ell\geq 5$.

Suppose $f$ has a congruence at $0\mod\ell$.  Then by Proposition~\ref{voacprop}, $f$ is a low point of its Tate cycle and by Lemma~\ref{tate_cycle_basics_lemma} the other low point has filtration $\omega(f)+4$.  Hence $(\Theta^{\frac{\ell+3}{2}}f)/f \in \overline{M}_{4}$ by Proposition~\ref{narrowingprop}.  We compute
\begin{align*}
f &\equiv q^{2\delta}\theta_{0}(z)\prod(1-q^{2n})^{-2} + O\left( q^{\ell + 2\delta} \right) \mod\ell\\
&\equiv q^{2\delta} +4q^{2\delta+1} + 9q^{2\delta+2} + 20q^{2\delta+3}+\cdots\mod\ell\\
f^{-1} &\equiv q^{-2\delta} -4q^{-2\delta+1} + 7q^{-2\delta+2} -12q^{-2\delta+3}+\cdots\mod\ell
\end{align*}
and
\begin{align*}
\Theta^{\frac{\ell+3}{2}}f \equiv (2\delta)^{\frac{\ell+3}{2}}q^{2\delta} +4 (2\delta+1)^{\frac{\ell+3}{2}}q^{2\delta+1} +9 (2\delta+2)^{\frac{\ell+3}{2}}q^{2\delta+2} +20 (2\delta+3)^{\frac{\ell+3}{2}}q^{2\delta+3} + \cdots \mod\ell.
\end{align*}
Hence we compute
\begin{align}
\left(\Theta^{\frac{\ell+3}{2}}f\right) f^{-1} \equiv 
(2&\delta)^{\frac{\ell+3}{2}} + \left( -4(2\delta)^{\frac{\ell+3}{2}} + 4(2\delta+1)^{\frac{\ell+3}{2}} \right)q\nonumber\\
&+ \left( 7(2\delta)^{\frac{\ell+3}{2}} -16(2\delta+1)^{\frac{\ell+3}{2}} +9(2\delta+2)^{\frac{\ell+3}{2}} \right)q^{2}\nonumber\\\label{cphi1}
&+ \left( -12(2\delta)^{\frac{\ell+3}{2}} +28(2\delta+1)^{\frac{\ell+3}{2}} -36(2\delta+2)^{\frac{\ell+3}{2}} +20(2\delta+3)^{\frac{\ell+3}{2}} \right)q^{3}\\&+\cdots \mod\ell.\nonumber
\end{align}
Recalling our basis (\ref{EFbasis}), we conclude
\begin{align}
\left(\Theta^{\frac{\ell+3}{2}}f\right) f^{-1} \equiv  
(2\delta)^{\frac{\ell+3}{2}} E^{2} &+ \left( 12(2\delta)^{\frac{\ell+3}{2}} + 4(2\delta+1)^{\frac{\ell+3}{2}} \right) EF\nonumber\\
&+ \left( -9(2\delta)^{\frac{\ell+3}{2}} + 16(2\delta+1)^{\frac{\ell+3}{2}} + 9(2\delta+2)^{\frac{\ell+3}{2}} \right) F^{2}\label{cphi2}.
\end{align}
Multiplying the coefficients of $q^{3}$ in both (\ref{cphi1}) and (\ref{cphi2})  by $12^{\frac{\ell+3}{2}}$ leads to
\begin{align}
0 &\equiv 100(-1)^{\frac{\ell+3}{2}} -84(11)^{\frac{\ell+3}{2}}
-36(23)^{\frac{\ell+3}{2}} +20(35)^{\frac{\ell+3}{2}} \mod\ell\nonumber\\
&\equiv 100 \leg{-1}{\ell} - 10164 \leg{11}{\ell} -19044\leg{23}{\ell} +24500\leg{35}{\ell}\mod\ell\label{cphi2 leq eq}\\
&\equiv \pm100 \pm 10164 \pm 19044 \pm 24500 \mod \ell.\label{cphi2 pm eq}
\end{align}
The only primes $\ell\geq 5$ satisfying (\ref{cphi2 pm eq}) are $5, 13, 19, 31, 59, 97, 131, 601$, and $6701$.  It is easily checked that only $\ell=5$ satisfies (\ref{cphi2 leq eq}).  That is, we have recovered the congruence (\ref{andrews2}) and proved there are no others.
\end{proof}

\bibliographystyle{plain}
\bibliography{/Users/Michael/Documents/bibliog}
\end{document}